\documentclass[11pt,oneside]{amsart}

\usepackage[left=1in,right=1in,top=1in,bottom=1in,marginparwidth=0.8in]{geometry}

\usepackage{microtype}  

\usepackage{mathtools}
\usepackage{bbm}  

\swapnumbers

\usepackage{hyperref}

\usepackage{amssymb,amsfonts,amsmath,amsthm}
\usepackage[mathscr]{eucal}
\usepackage[alphabetic]{amsrefs}
\usepackage{stmaryrd}
\usepackage{colonequals}

\usepackage[ps,matrix,arrow,curve,cmtip]{xy}

\title{Cofreeness of the Lubin-Tate deformation ring}
\author{Charles Rezk}
\date{ \today}
\address{Department of Mathematics \\
University of Illinois Urbana-Champaign \\ 
Urbana, IL}
\email{rezk@illinois.edu}

\numberwithin{equation}{section}

\makeatletter
  \let\c@subsection\c@equation
\makeatother

\theoremstyle{plain}   

\newtheorem{thm}[subsection]{Theorem}
\newtheorem*{thm*}{Theorem}
\newtheorem{prop}[subsection]{Proposition}
\newtheorem*{prop*}{Proposition}
\newtheorem{cor}[subsection]{Corollary}
\newtheorem*{cor*}{Corollary}
\newtheorem{lemma}[subsection]{Lemma}
\newtheorem*{lemma*}{Lemma}

\newtheorem*{claim*}{Claim}

\theoremstyle{remark}
\newtheorem{rem}[subsection]{Remark}    
\newtheorem*{rem*}{Remark}
\newtheorem{exam}[subsection]{Example}
\newtheorem*{exam*}{Example}

\theoremstyle{plain}

\begin{document}


\newcommand{\margnote}[1]{\mbox{}\marginpar{\tiny\hspace{0pt}#1}}

\def\lambada{\lambda}


\newcommand{\id}{\operatorname{id}}
\newcommand{\colim}{\operatorname{colim}}
\newcommand{\llim}{\operatorname{lim}}
\newcommand{\Cok}{\operatorname{Cok}}
\newcommand{\Ker}{\operatorname{Ker}}
\newcommand{\Image}{\operatorname{Im}}
\newcommand{\op}{{\operatorname{op}}}
\newcommand{\Aut}{{\operatorname{Aut}}}
\newcommand{\End}{{\operatorname{End}}}
\newcommand{\Hom}{{\operatorname{Hom}}}

\newcommand*{\ra}{\rightarrow}
\newcommand*{\lra}{\longrightarrow}
\newcommand*{\xra}{\xrightarrow}
\newcommand*{\la}{\leftarrow}
\newcommand*{\lla}{\longleftarrow}
\newcommand*{\xla}{\xleftarrow}

\newcommand{\ho}{\operatorname{ho}}
\newcommand{\hocolim}{\operatorname{hocolim}}
\newcommand{\holim}{\operatorname{holim}}

\newcommand*{\realiz}[1]{\left\lvert#1\right\rvert}
\newcommand*{\len}[1]{\left\lvert#1\right\rvert}
\newcommand*{\Len}[1]{\left\lVert#1\right\rVert}
\newcommand{\set}[2]{{\{\,#1\,\mid\,#2\,\}}}
\newcommand{\bigset}[2]{\left\{\,#1\;\middle|\;#2\,\right\}}
\newcommand*{\tensor}[1]{\underset{#1}{\otimes}}
\newcommand*{\pullback}[1]{\underset{#1}{\times}}
\newcommand*{\powser}[1]{[\![#1]\!]}
\newcommand*{\laurser}[1]{(\!(#1)\!)}
\newcommand{\ndiv}{\not|}
\newcommand{\pairing}[2]{\langle#1,#2\rangle}

\newcommand{\lrtensor}[3]{\,{\mathstrut}^{#1}\!\!\otimes_{#2}\!{\mathstrut}^{#3}}
\newcommand{\ltensor}[2]{\,{\mathstrut}^{#1}\!\!\otimes_{#2}}
\newcommand{\rtensor}[2]{\otimes_{#1}\!\!{\mathstrut}^{#2}}

\newcommand{\F}{\mathbb{F}}
\newcommand{\Z}{\mathbb{Z}}
\newcommand{\N}{\mathbb{N}}
\newcommand{\R}{\mathbb{R}}
\newcommand{\Q}{\mathbb{Q}}
\newcommand{\C}{\mathbb{C}}

\newcommand{\point}{{\operatorname{pt}}}
\newcommand{\Map}{\operatorname{Map}}
\newcommand{\eev}{\wedge}
\newcommand{\sm}{\wedge} 

\newcommand{\bbone}{\mathbbm{1}}

\newcommand*{\mc}{\mathcal}
\newcommand*{\msc}{\mathscr}
\newcommand*{\mf}{\mathfrak}
\newcommand*{\mr}{\mathrm}
\newcommand*{\mb}{\mathbb}
\newcommand*{\ul}{\underline}
\newcommand*{\ol}{\overline}
\newcommand*{\wt}{\widetilde}
\newcommand*{\wh}{\widehat}
\newcommand*{\mtt}{\mathtt}
\newcommand*{\ms}{\mathsf}
\newcommand*{\mbf}{\mathbf}

\newcommand{\dfn}{\textbf}

\def\noloc{\;{:}\,}

\newcommand{\defeq}{\colonequals}

\newcommand{\forcepar}{\mbox{}\par}


\renewcommand{\phi}{\varphi}

\newcommand{\Sym}{\operatorname{Sym}}
\newcommand{\Hi}{\ms{H}_\infty}
\newcommand{\CP}{\mathbb{CP}}

\newcommand{\Spec}{\operatorname{Spec}}
\newcommand{\Spf}{\operatorname{Spf}}

\newcommand{\Norm}{\operatorname{Norm}}

\newcommand{\Perf}{\ms{Perf}}

\newcommand{\Mod}{\operatorname{Mod}}
\newcommand{\Alg}{\operatorname{Alg}}

\newcommand{\Sp}{\ms{Sp}}
\newcommand{\T}{\mathbb{T}}
\newcommand{\Sphere}{\mathbb{S}}

\newcommand{\FG}{\ms{FG}}
\newcommand{\cDef}{\ms{Def}}
\newcommand{\fDef}{\ms{Def}^{\mr{Fr}}}
\newcommand{\tDef}{\ms{Def}^{\mr{tr}}}

\newcommand{\QCoh}{\operatorname{QCoh}}

\newcommand{\Id}{\operatorname{Id}}

\newcommand{\Cat}{\ms{Cat}}
\newcommand{\rgCat}{\ms{Cat}^{\mr{rg}}}

\newcommand{\oDef}{\ol{\ms{Def}}}
\newcommand{\iso}{\mathrm{iso}}
\newcommand{\cart}{\mathrm{cart}}
\newcommand{\isog}{\mathrm{isog}}
\newcommand{\triv}{\mathrm{triv}}

\newcommand{\Set}{\ms{Set}}
\newcommand{\CAlg}{\ms{CAlg}}

\newcommand{\Fr}{\operatorname{Fr}}
\newcommand{\tr}{\mr{tr}}

\newcommand{\univ}{\mathrm{univ}}
\newcommand{\ob}{\operatorname{ob}}
\newcommand{\mor}{\operatorname{mor}}

\renewcommand{\O}{\mc{O}}
\newcommand{\m}{\mathfrak{m}}

\newcommand{\Witt}{\mathbb{W}}
\newcommand{\gh}{\operatorname{gh}}

\newcommand{\cR}{\wh{\mc{R}}}

\begin{abstract}
We give a proof of the cofreeness of the
Lubin-Tate deformation ring, by generalizing earlier results by Matt
Ando and Yifei Zhu about $\Hi$-orientations to the context of power
operations for Morava $E$-theory.
\end{abstract}

\maketitle

\section{Introduction}

The purpose of this note is to give a proof
of what has 
been called the ``cofreeness of the Lubin-Tate ring'', using known
properties of power operations in Morava $E$-theory.

To any 1-dimensional commutative formal groups
$\Gamma/\kappa$ of finite height $h$ over a perfect field $\kappa$ of
characteristic $p$, Lubin and Tate \cite{lubin-tate-formal-moduli}
identified its \emph{universal deformation} $\Gamma^\univ$ as a formal group over a
ring $\O\approx \Witt_p\kappa\powser{u_1,\dots,u_{h-1}}$, the
\dfn{Lubin-Tate ring}.   As Morava observed
\cite{morava-noetherian-loc-cobordism}, this formal group is
associated to a complex-periodic generalized cohomology theory $E$ with
$\pi_0E=\O$, called \dfn{Morava $E$-theory}, which
Goerss-Hopkins-Miller \cite{goerss-hopkins-moduli-spaces} showed is
represented by a commutative ring in spectra.

\subsection{$\Hi$-orientation of Morava $E$-theory}

Recall that the forgetful functor $\CAlg(\Sp)\ra \Sp$ of
$\infty$-categories is \emph{monadic}: the underlying functor of the
monad is the symmetric algebra functor $\Sym(X)\approx  \bigoplus_{m\geq0}
(X^{\otimes m})_{\Sigma_m}$.   This monad descends to a monad on the
homotopy category $h\Sp$, and thus is associated to a 1-category of
algebras $\Hi(\Sp)$, called the category of \dfn{$\Hi$-ring spectra}
(see \cite{bmms-h-infinity-ring-spectra}).
There are evident forgetful functors
\[
  h\CAlg(\Sp) \ra \Hi(\Sp) \rightarrowtail \CAlg(h\Sp),
  \]
but both are far from being equivalences in general.

Nonetheless, one might want to describe all possible maps in
$\Hi(\Sp)$ between two commutative ring spectra.  Recall
that since $E$ is complex orientable, restriction along
$\Sigma^{\infty-2}\CP^\infty\ra MU$ induces a bijection
\[
  \Hom_{\CAlg(h\Sp)}( MU, E) \xra{\sim} \set{f\in
    \wt{E}^2\CP^\infty}{f|_{\CP^1}=1}
\]
between ring maps up-to-homotopy from $MU$ and  coordinates on the
universal deformation $\Gamma^\univ/\O$ (more precisely, to
coordinates up to rescaling by $\O^\times$).
\begin{thm}[Ando \cite{ando-isogenies-and-power-ops}; Zhu
  \cite{zhu-norm-coherence}] \label{thm:h-infty-mu}
  Restriction along $\Sigma^{\infty-2}\CP^\infty\ra MU$ and projection
  along $\O\ra \kappa$ induce a bijection
  \[
    \Hom_{\Hi(\Sp)}(MU, E) \xra{\sim} \set{\ol{f}\in  \kappa\otimes_\O
      \wt{E}^2\CP^\infty}{\ol{f}|_{\CP^1}=1}. 
  \]
  That is, every  coordinate of the formal group $\Gamma/\kappa$ lifts
  uniquely to a coordinate on $\Gamma^\univ/\O$, which is furthermore realized by
  a unique $\Hi$-map $MU\ra E$.
\end{thm}

In addition to this, these results come with a precise
characterization of the coordinates on 
$\Gamma^\univ/\O$ which correspond to $\Hi$-complex orientations:
these are the \dfn{Ando coordinates}, or in
the terminology of \cite{zhu-norm-coherence}  the \dfn{norm
  coherent coordinates}.

There are analogous results classifying $\Hi$-orientations of $E$ by
the bordism spectra $MSU$ and $MU\langle 6\rangle$
\cite{ando-hopkins-strickland-h-infinity}.

\subsection{Power operations in Morava $E$-theory and $\Hi$-maps}

Ando's original proof of \eqref{thm:h-infty-mu} is phrased in terms of
\emph{power 
  operations}.  
Subsequent work (including
\cite{strickland-finite-subgroups-of-formal-groups},
\cite{strickland-morava-e-theory-of-symmetric}, \cite{ando-hopkins-strickland-h-infinity},
\cite{rezk-congruence-condition}) have advanced our understanding of
power operations for Morava $E$-theory.

What we need can be summarized as follows.
\begin{enumerate}
\item There is a monad $\T$ on the 1-category $\Mod_{\O}$ of modules
  over the Lubin-Tate ring, which approximates $\pi_0\Sym$ in
  $\Mod_{E}(\Sp_{K(h)})$, in the sense that there is a natural map
  \[
    \T(\pi_0M) \ra \pi_0\Sym(M), \qquad M\in \Mod_E(\Sp_{K(h)}), 
  \]
  with the property that if  if $M$ is a
  $K(h)$-localization of a free $E$-module (that is, of a direct sum
  $\bigoplus E$), then
  $\T(\pi_0M)^{\eev}_{\m} \xra{\sim} \pi_0\Sym(M)$ (completion at the
  maximal ideal $\m\subset \O$).

\item Let $\Alg_\T$ denote the category of algebras for the monad
  $\T$.   Then $\pi_0\colon \CAlg_E(\Sp_{K(h)})\ra \CAlg_{\O}$ lifts
  canonically to a functor
  \[
    \pi_0\colon \CAlg_E(\Sp_{K(h)})\ra \Alg_{\T}.
  \]
  Furthermore, this functor evidently factors
  \[
    \CAlg_E(\Sp_{K(h)})\ra \Hi(\Mod_E(\Sp_{K(h)}))\ra \Alg_\T,
  \]
  through the 1-category of $\Hi$-rings in $K(h)$-local $E$-modules.

\item An immediate consequence of (1) and (2) is that, if $A\in
  \Hi(\Mod_E(\Sp_{K(h)}))$ is 
  such that its underlying $E$-module is the $K(h)$-localization of a
  free module, then
  \[
    \Hom_{\Hi(\Mod_E(\Sp_{K(h)}))} (A,B) \xra{\sim}
    \Hom_{\Alg_\T}(\pi_0 A,\pi_0B). 
  \]

  \item
In particular, if $A\in \Hi(\Sp)$ such that its $K(h)$-localized
$E$-homology $E^\vee_*A$ is concentrated in even degree and is such
that $E^\vee_0A$ is the $\m$-adic completion of a free $\O$-module, we
have that
\[
  \Hom_{\Hi(\Sp)}(A,E) \xra{\sim} \Hom_{\Alg_\T}(E^\vee_0A, \O).  
  \]
\end{enumerate}

\subsection{Cofreeness and its proofs}

Thus \eqref{thm:h-infty-mu} is a consequence of the following, which
can be thought of as a purely algebraic manifestion of the same idea.

\begin{thm}[Cofreeness of $\O$]\label{thm:cofreeness-of-O}
  For any object $R\in \Alg_\T$, the map
  \[
    \Hom_{\Alg_\T}(R,\O) \ra \Hom_{\CAlg_\kappa}(R/\m R,\kappa),
  \]
  obtained from the forgetful functor $\Alg_\T\ra \CAlg_\O$ and
  reducing along $\m\subset \O$, is a bijection.
\end{thm}
This was proved by Burklund, Schlank, and Yuan
\cite{burklund-schlank-yuan-chromatic-nullstellensatz}*{Thm.\ 3.4}, as
a component of their proof of the ``chromatic Nullstellensatz
theorem''.  
Their formulation of cofreeness is slightly more general than this, using the observation
\cite{lurie-ell2}*{\S5} that the construction of Morava $E$-theory for
$\Gamma/\kappa$ can be extended to functor
\[
  \ms{E}\colon  \Perf_\kappa \ra \CAlg(\Sp_{K(h)})
\]
from the category of perfect $\kappa$-algebras, so that
$\ms{E}(\kappa)=E$, and so that $\O(A)\defeq \pi_0\ms{E}(A)$ is
equipped with a surjective homomorphism $\O(A)\ra A$.
\begin{thm}[Cofreeness of $\O(A)$]\label{thm:cofreeness-of-OA}
For any object $R\in \Alg_\T$, and any $A\in Perf_\kappa$, 
\[
  \Hom_{\Alg_\T}(R, \O(A)) \xra{\sim} \Hom_{\CAlg_\kappa}(R/\m R, A)
\]
obtained from the forgetful functor $\Alg_\T\ra \CAlg_\O$ and
projecting along $\O(A)\ra A$  is a bijection. 
\end{thm}

The proof of these in
\cite{burklund-schlank-yuan-chromatic-nullstellensatz} relies on
previous work on power operations and the category of $\Alg_\T$.
It also makes use of some delicate calculations of power
operations due to Hahn \cite{hahn-thesis}, as well a transchromatic
induction argument.

As it happened, 
I had also found a proof of \eqref{thm:cofreeness-of-O}  a couple of
years prior, and had  
announced the result at a conference.\footnote{At the midwest topology seminar at
  University of Chicago in October 2019.}  However, I never got around
to writing up the proof then.
The purpose of this note is
to make that  proof available.

The proof given here is very
different than the one by Burklund, Schlank, and Yuan.  It may be
summarized as follows: \emph{The 
proof by Ando and Zhu of
\eqref{thm:cofreeness-of-O} can be essentially entirely rephrased 
in the setting of the algebra of  power operations}.  That is, not
merely the statement of \eqref{thm:cofreeness-of-O}, but also its
proof, can be given entirely using properties of the category
$\Alg_\T$, and by following the outline first given by Ando.

We note that Akhil Mathew has also given a proof of cofreeness.  His
proof has some correspondence to the one given here, but it also
makes essential use of the theory of animated $p$-derivations, \textit{a la}
\cite{bhatt-lurie-prismatization}*{App.\ A}.
Yet another proof of cofreeness is given in
\cite{rezk-witt-filtration}.  

\subsection{Idea of the proof}

In the original paper by Ando and its generalization by Zhu (and so also in
this note), an important role is
played by the connection between \emph{additive} power operations and
\emph{isogenies} 
between deformations of $\Gamma/\kappa$.  

The forgetful functor from $\T$-algebras to
$\O$-algebras factors as
\[
  \Alg_\T \ra \QCoh^{\Fr}(\cDef)\subset \QCoh(\cDef) \ra \Alg_\O.
\]
The category $\QCoh(\cDef)$ encodes information about \emph{additive} power
operations.  The forgetful functor lands in a full subcategory of
objects which satisfy the \emph{Frobenius congruence}.  This
congruence can be characterized as saying that a certain additive
power operation $Q$ satisfies an identity of the form
\[
  Q(x) = x^p+ p\theta(x),
\]
where $\theta$ is a \emph{non-additive $p$-derivation}.
(We'll use a different but
equivalent  formulation of the Frobenius congruence in this paper, see
\S\ref{subsec:frobenius-congruence}.)  
Furthermore, we have that
$\Hom_{\Alg_\T}(R,S)\xra{\sim}\Hom_{\QCoh(\cDef)}(R,S)$ whenever the
underlying $\O$-algebra of $S$ is $p$-torsion free.  

The category $\QCoh(\cDef)$ can be described purely in terms of the
formal group $\Gamma/\kappa$: it is here called the category of
\emph{quasicoherent sheaves of rings} on the stack $\cDef$ of
deformations and isogenies between them.  See
\S\ref{sec:quasicoherent-sheaves}. 

For the proof of cofreeness, the most relevant piece of structure is
the \emph{$p$th-power Adams operation}.  In particular, we will show
(\S\ref{subsec:adams-operation}) that to each object $\ms{R}$ in
$\QCoh^{\Fr}(\cDef)$, its underlying $\O$-algebra 
$R$ carries a
natural ring-endomorphism
$  \psi_{\ms{R}}\colon R\ra R$ which is a lift of $p^h$-power Frobenius:
\[
  \psi_{\ms{R}}(x) \equiv x^{p^h} \mod \m R.
\]
In particular, this defines a natural operation on $\pi_0$ of
$K(h)$-local commutative $E$-algebras, which generalizes the classical
$p$th power
Adams operation on $p$-completed complex $K$-theory.

The structure of the proof is as follows.  First, we can reduce to the case when the
underlying $\O$-algebra of $R\in \Alg_\T$ is a polynomial ring, as
this is so for the universal example (the free $\T$-algebra on one
generator).  As $\O$ is $p$-torsion free, we can work entirely 
in the category  $\QCoh^{\Fr}(\cDef)$.
Given this, the argument has two steps.
\begin{enumerate}
\item We show that any $\O$-algebra
  homomorphism $\mu\colon R\ra \O$ is approximated \emph{uniquely}
  modulo $\m\subset \O$
  by one which  
  commutes with Adams operations, i.e., such that $\mu \psi_{R}=
  \psi_{\O}\mu$.
\item We show that any $\O$-algebra map $\mu\colon R\ra \O$ which
  commutes with Adams operations is in fact a morphism in
  $\QCoh(\cDef)$.
\end{enumerate}

\subsection{Comparison to the $\Hi$-argument}

This two-step argument is precisely parallel to those given by Ando
and Zhu.  In their case, $R=E_0^{\eev}MU$, which is an $\m$-adic
completion of a polynomial ring.    This object corepresents
the functor $F$ which sends an $\O$-algebra $A$ to the set of
\emph{coordinates (up to units)} on the basechange
$\Gamma^{\mr{univ}}_A$ of the universal deformation to $A$.  Power  
operations in $R$  encode of such
coordinates with respect to isogenies of formal groups: an 
isogeny $f\colon G\ra G'$ sends a coordinate $x$ on $G$ to its norm
$y\defeq \Norm_f(x)$ with respect to $f$.  See
\cite{zhu-norm-coherence} for details.

A $\T$-algebra map $R\ra \O$ (and thus an $\Hi$-map $MU\ra E$)
corresponds to a \emph{norm-coherent} coordinate on
$\Gamma^{\mr{univ}}/\O$, i.e., a natural choice of coordinate $x_G$ for each
deformation $G$ such that
\[
  x_{G'} = \Norm_f(x_G) \quad \text{for every isogeny $f\colon G\ra G'$.}
\]
Ando and Zhu lift a coordinate $\ol{x}$ on $\Gamma/\kappa$ to a
norm-coherent coordinate $x$ on the universal deformation
$\Gamma^{\mr{univ}}/\O$ in two steps: (i) produce a lift $x$ which is
unique with the property of being compatible with norms along the
$p$th power isogeny, then (ii) show 
that any such lift $x$  is compatible with norms along arbitrary isogenies.

\subsection{Structure of this paper}

In \S\ref{sec:frobenius-lifting} we recall standard ideas about
approximating ring homomorphisms by ones which commute with given
lifts of Frobenius.

In \S\ref{sec:quasicoherent-sheaves} we define the category
$\QCoh(\cDef)$ of quasicoherent sheaves of rings on the stack $\cDef$
of deformations.  This is largely a summary of material from
\cite{rezk-congruence-condition}, focusing on what we need here.
Using this we define the \emph{Adams operations} which are used in the
proof of the main theorem.

The following section \S\ref{sec:frob-lift-qcoh} is the heart of the
argument, where we apply Frobenius lifting to $\QCoh(\cDef)$.  The
proof of the cofreeness of $\O$ \eqref{thm:cofreeness-of-O} is nearly
immediate, and is given in \S\ref{sec:proof-thm}.  In that section we
also prove the extension \eqref{thm:cofreeness-of-OA}, using some
additional facts from
\cite{burklund-schlank-yuan-chromatic-nullstellensatz}.

\section{Frobenius lifting}
\label{sec:frobenius-lifting}

In the proof of the theorem, we will need to be able to approximate ring
homomorphisms (up to a congruence) by ones which commute with given
``lifts of Frobenius''.  This is of course a very well-known idea.  In
this section we give a careful statement and proof of the version of
this that we will use.

Fix a prime $p$ and an integer $h\geq 1$.  We define a category
$\mc{C}_{p^h}$ of ``rings equipped with a lift of $p^h$-Frobenius''.
Explicitly,  objects are triples $(R,I_R,\psi_R)$, where $R$ is a
commutative ring, $I_R\subseteq R$ is an ideal such that $p\in I_R$, 
and $\psi_R\colon R\ra R$ is a ring endomorphism such that
\[
\psi_R(x) \equiv x^{p^h} \mod I_R \qquad \text{for all $x\in R$.}
\]
A morphism $\mu\colon (R,I_R,\psi_R)\ra (S,I_S,\psi_S)$ in $\mc{C}_{p^h}$ is
a ring homomorphism $\mu\colon R\ra S$ such that $\mu(I_R)\subseteq
I_S$ and $\mu\psi_R=\psi_S\mu$.

Say that an object $(R,I_R,\psi_R)$ of $\mc{C}_{p^h}$ is \dfn{complete} if
$R\xra{\sim} \lim_k R/I_R^k$, and is \dfn{perfect} if $\psi_R$ is an
isomorphism such that $\psi_R(I_R)=I_R$.

\begin{rem}
For any such object $(R,I_R,\psi_R)$, the map $\psi_R$ is a lift
of the $p^h$-power Frobenius on $R/I_R$.  Furthermore, if the object
is perfect, then $R/I_R$ is a perfect ring.
\end{rem}

\begin{rem}
  Note that for an object $(R,I_R,\psi_R)$ in $\mc{C}_{p^h}$, the map
  $\psi_R\colon R\ra R$ is provides a natural  endomorphism in the category
  $\mc{C}_{p^h}$,
  since $x\in I_R$ implies $\psi_R(x)\in I_R$.  Furthermore,  an object is
  perfect iff this endomorphism is an automorphism.
\end{rem}

The following gives criteria which construct a unique morphism of $\mc{C}_{p^h}$
approximating a given  ring homomorphism.  
\begin{prop}[Frobenius lifting]\label{prop:frobenius-lifting}
  Let
  \[
  \iota_R\colon (A,I_A,\psi_A)\ra (R,I_R,\psi_R)\qquad\text{and}
  \qquad \iota_S\colon (A,I_A,\psi_A)\ra (S,I_S,\psi_S)
\]
be morphisms in $\mc{C}_{p^h}$, and suppose that $I_R=\iota_R(I_A)R$ and that  
$(S,I_S,\psi_S)$ is
complete and perfect.  Suppose $\mu\colon R\ra S$ be a ring homomorphism
such that $\mu\iota_R=\iota_S$ (i.e., $\mu$ is a map of $A$-algebras).
Then
there exists a unique morphism $\wh\mu\colon (R,I_R,\psi_R)\ra
(S,I_S,\psi_S)$ in $\mc{C}_{p^h}$ such that
\[
\wh\mu\iota_R=\iota_S\qquad \text{and} \qquad \wh\mu\equiv \mu \mod I_S.
\]
\end{prop}
\begin{proof}
  Suppose given a $\mu$ as in the statement.  Note that
  $\mu(I_R)=\mu(\iota_R(I_A)R)\subseteq I_S$ since $\mu$ is an $A$-algebra map and
  $I_R=\iota_R(I_A)R$.  Define the function
  \[
  \mu'\defeq \psi_S^{-1}\mu\psi_R \colon R\ra S.
\]
Clearly
\begin{itemize}
\item $\mu'$ is a ring homomorphism,
\item $\mu'\iota_R=\iota_S$, since
  \[
  \mu\psi_R \iota_R = \mu \iota_R\psi_A = \iota_S\psi_A=\psi_S\iota_S.  
\]
\item $\mu'(I_R)\subseteq I_S$, since $\psi_R(I_R)\subseteq I_R$ and
  $\psi_S^{-1}(I_S)=I_S$.
\item $\mu'(x)\equiv \mu(x)\mod I_S$ for all $x\in R$, since
  \[
  \mu\psi_R(x)\equiv \mu(x^{p^h})=\mu(x)^{p^h} \mod \mu(I_R)\subseteq
  I_S\qquad \text{and}
  \qquad
  \psi_S\mu(x) \equiv \mu(x)^{p^h} \mod I_S,
\]
together with the fact that $\psi_S^{-1}(I_S)=I_S$.  
\end{itemize}
We will prove the following Claim, which asserts that ``conjugation by
$\psi$ is a contraction operator'' on ideal preserving $A$-algebra
maps.

\begin{enumerate}
\item [\textit{Claim.}]
  Suppose $\mu_1,\mu_2\colon R\ra S$ are ring homomorphisms such that
  \[
  \mu_i\iota_R=\iota_S, \qquad \mu_i(I_R)\subseteq I_S, \qquad i=1,2.
\]
If $\mu_1(x)\equiv \mu_2(x)\mod I_S^r$ for some $r\geq 1$ and all
$x\in R$, then
$\mu_1'(x)\equiv \mu_2'(x)\mod I_S^{r+1}$ for all $x\in R$.
\end{enumerate}
Given the Claim, since $S$ is complete with respect to $I_S$ we may define
\[
\wh\mu(x)\defeq \lim_{r\to\infty} \psi_S^{-r}\mu \psi_R^r(x) =
\lim_{r\to\infty} \mu^{(r)}(x),
\]
where $\mu^{(r)}\defeq (\mu^{(r-1)})'$.
The function $\wh\mu\colon R\ra S$ clearly satisfies
\begin{itemize}
\item $\wh\mu\iota_R=\iota_S$,
\item $\wh\mu(I_R)\subseteq I_S$,
  \item $\wh\mu\psi_R = \psi_S\wh\mu$, since $\mu^{(r)}\psi_R\equiv
    \psi_S \mu^{(r)} \mod I_S^{r+1}$,
    \item $\wh\mu(x)\equiv \mu(x) \mod I_S$ for all $x\in R$,
    \end{itemize}
    so is the desired map.  Uniqueness is straightforward: given two
     solutions $\wh\mu_1,\wh\mu_2$ in $\mc{C}_{p^h}$ which agree modulo some
    $I_S^r$ with $r\geq1$, the above argument implies they agree
    modulo $I_S^{r+1}$, so completeness of $S$ implies $\wh\mu_1=\wh\mu_2$.

\textit{Proof of Claim.}  Given $\mu_1,\mu_2$ as in the Claim, write
\[
\mu_2(x)=\mu_1(x)+\alpha(x),\qquad \psi_R(x)= x^{p^h}+\beta(x),
\]
defining functions $\alpha\colon R\ra I_S^r$ and $\beta\colon R\ra
I_R$.  Note that $\alpha$ is a map of $A$-modules, 
since for any $a\in A$ and $x\in R$ we have 
\[
\alpha(\iota_R(a)x)= \mu_2(\iota_R(a)x)-\mu_1(\iota_R(a)x)=
\iota_S(a)\mu_2(x)-\iota_S(a)\mu_1(x)=\iota_S(a)\alpha(x).
\]
Thus, $\alpha(I_R)=\alpha(\iota_R(I_A)R)\subseteq I_S\alpha(R)\subseteq
I_S^{r+1}$, and so
$\alpha\beta(R)\subseteq
I_S^{r+1}$.

For $x\in R$ we have
\begin{align*}
  \mu_2\psi_R(x)
  &= \mu_2\bigl( x^{p^h}+\beta(x)\bigr)
  \\
  &= \mu_2(x)^{p^h} + \mu_2\beta(x)
  \\
  &= \bigl(\mu_1(x)+\alpha(x)\bigr)^{p^h} + \mu_1\beta(x) +
    \alpha\beta(x)
  \\
  &= \mu_1(x)^{p^h} + \sum_{k=1}^{p^h-1}
    \binom{p^h}{k}\mu_1(x)^{p^h-k}\alpha(x)^k + \alpha(x)^{p^h} +
    \mu_1\beta(x) + \alpha\beta(x)
  \\
  &=\mu_1\bigl(x^{p^h}+\beta(x)\bigr) +
    \sum_{k=1}^{p^h-1}\binom{p^h}{k}\mu_1(x)^{p^h-k}\alpha(x)^k +
    \alpha(x)^{p^h}+\alpha\beta(x)
  \\
  &\in \mu_1\psi_R(x) + \sum_{k=1}^{p^h-1} p I_S^{rk} + I_S^{rp^h} +
    I_S^{r+1}
  \\
  &\subseteq \mu_1\psi_R(x) + I_S^{r+1},
\end{align*}
since $p\in I_S$ and $p^h>1$.  This proves the Claim.
\end{proof}

For instance, we immediately deduce the following well-known observation.
\begin{cor}
  Let $(S,I_S,\psi_S)$ be a complete and perfect object in
  $\mc{C}_{p^h}$.  Then for every $a\in S/I_S$ there exist a
  unique $\wt{a}\in S$ such that (i) $\wt{a}\equiv a\mod I_S$ and (ii)
  $\psi_S(\wt{a})=\wt{a}^{p^h}$.   
\end{cor}
\begin{proof}
  Apply the theorem with $(A,I_A,\psi_A)=(\Z,p\Z, \id)$ and
  $(R,I_R,\psi_R)=(\Z[x], p\Z[x], x\mapsto x^{p^h})$, using $\mu\colon
  \Z[x]\ra S$ such that $\mu(x)=a'$, where $a'\in S$ is any lift of $a$.
\end{proof}

\section{Quasicoherent sheaves on deformations}
\label{sec:quasicoherent-sheaves}

The material in the first part of this  section summarizes results from 
\cite{rezk-congruence-condition}*{\S11, \S12}, which themselves rely
heavily on work of Strickland
\cite{strickland-finite-subgroups-of-formal-groups}, 
\cite{strickland-morava-e-theory-of-symmetric}.  In particular, we
describe  a category $\QCoh(\cDef)$ of ``quasicoherent sheaves'' on a
``stack of deformations'' and a full subcategory $\QCoh^{\Fr}(\cDef)$
of objects which satisfy the ``Frobenius congurence''.  An explanation
for this is given in \cite{rezk-congruence-condition} (in whose
terminology $\QCoh(\cDef)$ is  the category of commutative 
monoid objects in $\mr{Comod}_{\mc{L}}$).

We fix a formal group $\Gamma$ of height $1\leq h<\infty$ over a
perfect field $\kappa$ of characteristic $p$.

Let $\cR$ denote the subcategory of commutative rings whose objects are 
complete local rings $A$ of characteristic $p$, and whose
morphisms $f\colon A\ra B$ are continuous homomorphisms (i.e., such
that $f^{-1}\m_A=\m_B$).

\subsection{Deformations}

A \dfn{deformation} of $\Gamma/\kappa$ to $A\in \cR$ is a triple
$\mc{G}=(G,i,\gamma)$, consisting of
\begin{enumerate}
\item a formal group $G$ defined over $A$,
\item a ring homomorphism $i\colon \kappa \ra A/\m_A$, and 
\item an isomorphism $\gamma\colon i^*\Gamma\xra{\approx} G_0$ of formal groups
  over $A/\m_A$, where $G_0$ denotes the base change of $G$ to
  $A/\m_A$.    
\end{enumerate}

For any $A\in \cR$ in which $p=0$ we write $\phi\colon A\ra A$ for the
$p$th power map.  Recall that for any formal group $G$ over such $A$
we have relative Frobenius maps
\[
  G  \xra{\Fr^d} (\phi^d)^*G,\qquad  d\geq0.
  \]
Furthermore, if $G$ is a
deformation of a height $h$-formal group, then $\Fr^d$ is an isogeny
of degree $p^{dh}$.

A \dfn{morphism of deformations} $f\colon \mc{G}=(G,i,\gamma)\ra
\mc{G}'=(G',i',\gamma')$ to $A$ is a homomorphism of formal groups $f\colon
G\ra G'$ such that (i) $i'=i\phi^d$, and (ii) the diagram of formal groups
over $A/\m_A$ 
\[\xymatrix{
  {i^*\Gamma} \ar[d]_{\gamma}^{\approx} \ar[r]^{\Fr^d}
  & {(i')^*\Gamma} \ar[d]^{\gamma'}_{\approx}
  \\
  {G_0} \ar[r]_{f_0}
  & {G_0'}
}\]
commutes.  Note that this implies that $f$ is an isogeny of degree
$p^d$, and we'll say that the morphism $f$ has \dfn{height} $d$.

We obtain a category $\cDef(A)$ of deformations.  We note the
following property.
\begin{prop}\label{prop:rigidity-of-deformations}
  For any two objects $\mc{G}, \mc{G}'$ in $\cDef(A)$ and
  $d\in \Z_{\geq0}$   there is at most one morphism $f\colon
 \mc{G}\ra \mc{G}'$ of height $d$.
\end{prop}
\begin{proof}
  If a homomorphism $f$ of deformations to $A$ is such that $f_0=0$, then
  $f=0$.  Apply this to the difference $f-f'$ of maps between deformations.
\end{proof}

In particular, between any two objects of $\cDef(A)$ there is at most
one isomorphism (which are precisely the maps of height $0$).  Thus we
can form a quotient category $\oDef(A)\defeq \cDef(A)/\sim$ by
identifying isomorphic objects, and $\cDef(A)\ra \oDef(A)$ is an
equivalence.

Also note that, given a deformation $\mc{G}=(G,i,\gamma)$ and a finite subgroup
scheme $H\leq G$ of rank $p^d$, there exists a unique morphism
$f\colon \mc{G}\ra 
\mc{G}'$ of deformations such that $H=\Ker f$, and the height of $f$
is $d$.

The assignment $A\mapsto \oDef(A)$ defines a functor $\cR\ra \Cat $.
\begin{prop}\label{prop:def-is-aff-gr-cat-scheme}
The functor $\oDef(A)$ is represented up to isomorphism (and so
$\cDef(A)$ is represented up to equivalence)  by an
``affine graded category scheme''.  That is, there exist objects
\[
\O_d,\qquad d\geq0,  \qquad \O=\O_0,
\]
and morphisms
\[
s_d,t_d\colon \O\ra \O_d, \qquad \nabla_{d,e}\colon \O_{d+e} \ra
\O_d\lrtensor{s_d}{\O}{t_d} \O_e,
\]
in $\cR$ such that
\[
\Hom_{\cR}(\O,A) = \operatorname{ob}\cDef(A), \qquad
\coprod_{d\geq0} \Hom_{\cR}(\O_d,A) = \operatorname{mor}\cDef(A), 
\]
and the maps $s_d,t_d,\nabla_{d,e}$ represent ``source'', ``target'',
and ``composition $fg$ where $f$ has height $d$ and $g$ has height
$e$'' respectively.

Furthermore,
\begin{itemize}
\item $\O$ is the Lubin-Tate ring, classifying isomorphism classes of
  deformations,  and
\item $\O_d$ is the $d$th Strickland ring, classifying subgroups of
  rank $p^d$ in a deformation.
\end{itemize}
These are complete local rings with maximal ideals  $\m_d\subset \O_d$
(and we write $\m=\m_0\subset \O)$, with $\O_d/\m_d\approx \kappa$.
The map $s_d$ exhibits $\O_d$ as a finitely generated free $\O$-module.
\end{prop}

\begin{rem}
 This structure  is over-determined in view of
  \eqref{prop:rigidity-of-deformations}: in fact, the the maps $s_d$
  and $t_d$ induce a surjective   
  homomorphism $\O\wh{\otimes}_{\mathbb{W}_p} \O\ra \O_{d}$
  \cite{burklund-schlank-yuan-chromatic-nullstellensatz}*{Prop.\
    3.39}, and the $\nabla_{d,e}$ are the unique maps compatible with 
  them.    
\end{rem}

\begin{rem}
  There are canonical isomorphisms $\O_d\approx
  E^0B\Sigma_{p^d}/T^{\mr{Tr}}_d$, where $J^{\mr{Tr}}_d$ is the ideal
  generated by transfers from subgroups of the form $\Sigma_i\times
  \Sigma_{p^d-i}$ where $0<i<p^d$
  \cite{strickland-morava-e-theory-of-symmetric}.  
\end{rem}

\subsection{Trivial deformations}
\label{subsec:trivial-deformations}

Say that a deformation $\mc{G}=(G,i,\gamma)$ to $A$ is \dfn{trivial} if (i)
$\F_p\subseteq A$, and (ii) $\gamma\colon i^*\Gamma\xra{\approx} G_0$
lifts to a (necessarily unique) isomorphism $\wt\gamma\colon \wt{i}^*
\Gamma\xra{\approx} G$ of formal groups over $A$, where $\wt{i}\colon
\kappa\ra A$ is the 
(necessarily unique) lift of $i\colon \kappa\ra A$.  Equivalently,
$\mc{G}$ is trivial iff it is isomorphic to a deformation of the form
$(\wt{i}^*\Gamma, i, \id)$.

We note that a deformation $\mc{G}\in \cDef(A)$ is trivial if and only
if its representing map $\chi_{\mc{G}}\colon \O\ra A$ vanishes on the maximal ideal $\m$.
Furthermore, a morphism $f\colon \mc{G}\ra \mc{G'}$ in $\cDef(A)$ is
between two trivial deformations if and only if its representing map
$\chi_{f}\colon \O_d\ra A$ vanishes on the maximal ideal $\m_d$.

\subsection{Frobenius morphisms}

Say that a morphism $f\colon \mc{G}=(G,i,\gamma)\ra \mc{G}'=(G',i',\gamma')$
is a \dfn{height $d$ Frobenius} if (i) $\F_p\subseteq A$ and (ii) the
homomorphism $f\colon G\ra G'$ of formal groups over $A$ factors as
\[
G \xra{\Fr^d} (i\phi^d)^*G \xra{\approx} G',
\]
i.e., iff $\Ker f\leq G$ is the ``canonical subgroup'' of degree $p^d$.
Equivalently, $f$ is a height $d$ Frobenius if it is isomorphic to a
morphism of the form $(G,i,\gamma)\ra ((\phi^d)^*G, i\phi^d,
(\phi^d)^*\gamma)$.

\begin{prop}\label{prop:nu-properties}
  There are surjective ring  homomorphisms
  \[
  \nu_d\colon \O_d\ra \O/p
\]
with the following properties.
\begin{enumerate}
\item  A height $d$ morphism of deformations is a
height $d$ Frobenius if and only if its representing map $\chi\colon \O_d\ra A$
factors through $\nu_d$.
\item $\nu_d s_d\colon \O\ra \O/p$ is the quotient map  and 
  $\nu_dt_d\colon \O\ra \O/p$ is the $p^d$th power map.
\item The quotient map factors as 
$\O_d \xra{\nu_d} \O/p \ra \O/\m=\kappa$.
\end{enumerate}
\end{prop}
\begin{proof}
The $\nu_d$ are the representing maps of the universal example of a
height $d$ Frobenius.  Statement (3) follows because any morphism
between trivial deformations is a Frobenius.
\end{proof}

\subsection{Quasicoherent sheaves}

By a \dfn{quasicoherent sheaf of rings} on $\cDef$, we mean data
$\ms{R}=(R,\iota_R,\nabla_{R,d})$ consisting of
\begin{enumerate}
\item a commutative ring $R$,
\item a ring homomorphism $\iota=\iota_R\colon \O\ra R$, and
\item ring homomorphisms $\nabla_{R,d}\colon R\ra
  \O_d\lrtensor{s}{\O}{\iota} R$ such that the diagrams
  \[\xymatrix@C=50pt{
    {R} \ar[r]^-{\nabla_{R,d}} \ar[d]_{\nabla_{R,d+e}}
    & {\O_d\lrtensor{s}{\O}{\iota} R} \ar[d]^{\id \otimes
      \nabla_{R,e}}
    & {\O} \ar[r]^{t_d} \ar[d]_{\iota_R}
    & {\O_d} \ar[d]^{\id\otimes\iota_R}
    \\
    {\O_{d+e}\lrtensor{s}{\O}{\iota} R} \ar[r]_-{\nabla_{d,e}\otimes\id}
    & {\O_d\lrtensor{s}{\O}{t} \O_e \lrtensor{s}{\O}{\iota} R}
    & {R} \ar[r]_-{\nabla_{R,d}}
    & {\O_d\lrtensor{s}{\O}{\iota} R}
    }\]
  commute.  
\end{enumerate}
That is, a quasicoherent sheaf is a ``comodule-algebra'' over the
bialgebra $(\O_d, s_d,t_e,\nabla_{d,e})$.  We write $\QCoh(\cDef)$ for
the evident category of quasicoherent sheaves, in which morphisms
$(R,\iota_R,\nabla_{R,d})\ra (S,\iota_S,\nabla_{S,d})$ are ring maps
compatible with 
all the structure.  Observe that $\QCoh(\cDef)$ is symmetric monoidal
via $\otimes_\O$.

\begin{exam}\label{exam:power-ops}
 If $B\in \CAlg_E(\Sp_{K(h)})$ is a $K(h)$-local algebra over Morava
 $E$-theory, then $\pi_0B$ naturally admits the structure of a
 quasicoherent sheaf, where $\nabla_{\pi_0B,d}$ is given by a power
operation: $\pi_0 B \ra \pi_0\ul{\Map}( \Sigma^\infty_+ B\Sigma_{p^d},
B) \approx E^0B\Sigma_{p^d} \otimes_\O \pi_0B \ra \O_d\otimes_\O
\pi_0B$.   As explained in \cite{rezk-congruence-condition}, this
construction factors as
\[
  \CAlg_E(\Sp_{K(h)})\xra{\pi_0} \Alg_\T\ra \QCoh(\cDef).
\]
\end{exam}

\subsection{Frobenius congruence}
\label{subsec:frobenius-congruence}

We say that a quasicoherent sheaf of rings $(R,\iota_R,\nabla_{R,d})$ on
$\cDef$ satisfies the 
\dfn{Frobenius congruence} if the diagrams
\[\xymatrix{
  {R} \ar[r]^-{\nabla_{R,d}} \ar[d]_{\phi^d}
  & {\O_d\lrtensor{s}{\O}{\iota} R} \ar[d]^{\nu_d\otimes \id}
  \\
  {R/p} \ar[r]_-{\approx}
  & {\O/p\lrtensor{s}{\O}{\iota} R}
}\]
commute for all $d\geq 0$, where $\phi^d\colon R\ra R/p$ is the map
induced by taking $p^d$th powers.  It is straightforward to see that
the Frobenius congruence holds iff the single diagram for $d=1$ commutes.

\begin{rem}
  The idea is that ``evaluation'' of the quasicoherent sheaf
  $(R,\iota_R,\nabla_{R,d})$ at the
  universal example of a height $d$ Frobenius produces a map
  $R\xra{(\nu_d\otimes \id)\nabla_{R,d}} R/p$.  The Frobenius
    congruence asserts that this map of rings should also be a $p^d$th power
    map.
\end{rem}

\begin{exam}
  The data $\ms{O}=(\O, \id, t_d)$ defines a quasicoherent sheaf satisfying
  the Frobenius congruence, since $t_d\nu_d=\phi^d$.  This is the unit
  object of the monoidal structure on $\QCoh(\cDef)$.
\end{exam}

\begin{exam}\label{exam:power-ops-frobenius}
  The construction of \eqref{exam:power-ops} always gives  a
  quasicoherent sheaf which satisfies the Frobenius congruence, i.e.,
  we have
  \[
  \CAlg_E(\Sp_{K(h)})\xra{\pi_0} \Alg_\T\ra
  \QCoh^{\Fr}(\cDef)\subseteq \QCoh(\cDef).
    \]
\end{exam}

\subsection{The $p$th power map}

Every formal group $G$ has a $p$th power endomorphism $p\colon G\ra G$, which
commutes with all homomorphisms, and thus can be regarded as an
endomorphism of the identity functor on the category of formal groups.
The $p$th power generally fails to be
an endomorphism of deformations.  However, it does give rise to a
natural transformation
\[
p\colon \Id\ra \Sigma \quad\text{of functors}\quad \cDef\ra \cDef,
\]
and hence a natural ``Adams operation'' $\psi_R \colon R\ra R$ for every
quasicoherent sheaf $(R, \iota_R,\nabla_{R,d})$, which for
quasicoherent sheaves which satisfy the Frobenius congruence is a lift
of $p^h$-power  
Frobenius.  We describe these below.

First note that since $\Gamma/\kappa$ is a formal group of height $h$,
its $p$th power map satisfies $p=\sigma\cdot \Fr^h$ for a unique
isomorphism $\sigma\colon (\phi^h)^*\Gamma\xra{\approx} \Gamma$ of
formal groups over $\kappa$.

For any deformation $\mc{G}=(G,i,\Gamma)$ to $R$ the $p$th power
map on  $G$ induces a
height $h$ morphism of deformations
\[
[p]\colon (G,i,\gamma) \ra (G, i\phi^h, \gamma\cdot i^*\sigma).
\]
Furthermore, any morphism of deformations fits in a commutative square 
\begin{equation}\label{eq:pth-power-natural-transformation}
  \xymatrix{
  {(G,i,\gamma)} \ar[r]^-{[p]} \ar[d]_{f}
  & {(G,i\phi^h,\gamma\cdot i^*\sigma)}  \ar[d]^{f}
  \\
  {(G',i',\gamma')} \ar[r]_-{[p]} 
  & {(G',i'\phi^h,\gamma'\cdot i^*\sigma)}
}
\end{equation}
since $fp=pf$ as maps of formal groups.

Thus, we have a functor $\Sigma\colon \cDef\ra \cDef$, given  on
objects by
\[
\Sigma(G,i,\gamma) \defeq (G,i\phi^h, \gamma\cdot i^*\sigma),
\]
and a natural transformation $[p]\colon \Id\ra \Sigma$, given on
objects by the maps $[p]$ above.  Note that the functor $\Sigma$ is an
equivalence.  

\begin{rem}
Recall that $\Aut(\Gamma/\kappa)$, the group of automorphisms of the
map $\Gamma\ra 
\Spec \kappa$ (i.e., formal group isomorphisms covering field
automorphisms), acts on the space of deformations, and in fact on the
category $\cDef$.  The functor $\Sigma$ is precisely the evaluation of
this action at an element of $\Aut(\Gamma/\kappa)$, namely that
represented by  $\sigma\colon (\phi^h)^*\Gamma\ra\Gamma$.
\end{rem}

\begin{exam}
  Recall that $\Gamma/\kappa$ is called a \dfn{Honda formal group} if
  $\kappa\subseteq \F_{p^h}$ and $p=\Fr^h$, i.e., iff $\sigma$
  is the identity map.  Thus, for the Honda formal group 
  $\Sigma\colon \cDef\ra \cDef$ is the identity functor.
\end{exam}

\begin{prop}\label{prop:pth-power-representing-properties}
  \forcepar
  \begin{enumerate}
    \item
  There are isomorphisms
  \[
  \psi_{\O_d}\colon \O_d\ra \O_d, \qquad d\geq0, \qquad
  (\psi_{\O_0}=\psi_\O)
\]
which represent the functor $\Sigma$.
\item
  There is a
homomorphism
\[
q\colon \O_h\ra \O
\]
in $\cR$ which represents the natural transformation $[p]$, in the
sense that if $\mc{G}$ is classified by $\chi\colon \O\ra R$, then the
map $[p]\colon \mc{G}\ra \Sigma\mc{G}$ is classified by $\chi q\colon
\O_h\ra \O$.
\item  We have that $qs_h=\id$ and $qt_h=\psi_\O$  and that the diagram 
  \[\xymatrix{
    {\O_{d+h}} \ar[r]^-{\nabla_{d,h}} \ar[d]_{\nabla_{h,d}}
    & {\O_d\lrtensor{s}{\O}{t} \O_h} \ar[d]^{\psi_{\O_d}\otimes q}
    \\
    {\O_h\lrtensor{s}{\O}{t} \O_d} \ar[r]_-{q\otimes \id}
    & {O_d}
  }\]
commutes for all $d$.
\item  The map $\psi_{\O_d}$ satisfies $\psi_{\O_d}(x)=x^{p^h} \mod\m_d$.
  \end{enumerate}
\end{prop}
\begin{proof}
  The existence of $\psi_{\O_d}$ and $q$ is immediate.  The square in
  (3) encodes the fact that $[p]$ is a natural transformation: the two
  ways of going around the square correspond to the two ways of going
  around \eqref{eq:pth-power-natural-transformation}.  
In view of
  \eqref{subsec:trivial-deformations}, to prove (4) it suffices to compute the
  action of the functor $\Sigma$ on trivial deformations, in which
  case we can reduce to the universal example of such, namely
  $(\Gamma, \id,\id)\in \cDef(\kappa)$.    
\end{proof}

\subsection{The Adams operation  $\psi_R$}
\label{subsec:adams-operation}

Given a quasicoherent sheaf $(R,\iota, \nabla_{R,d})$, we define a map
$\psi_R\colon R\ra R$ to be the composite of
\[
R\xra{\nabla_{R,h}} \O_h\lrtensor{s}{\O}{\iota} R \xra{q\otimes \id}
\O\lrtensor{s}{\O}{\iota} R=R.
\]
The key properties we need of this are the following.
\begin{prop}\label{prop:psi-properties}
  Let $(R,\iota_R,\nabla_{R,d})$ be a quasicoherent sheaf on $\cDef$.
  \begin{enumerate}
  \item We have that $\psi_R\iota_R= \iota_R\psi_{\O}$.  
  \item  the diagram
    \[ \xymatrix{
      {R} \ar[r]^-{\nabla_{R,d}} \ar[d]_{\psi_R}
      & {\O_d\lrtensor{s}{\O}{\iota} R}  \ar[d]^{\psi_{\O_d}\otimes \psi_R}
      \\
      {R} \ar[r]_-{\nabla_{R,d}}
      & {\O_d\lrtensor{s}{\O}{\iota} R}
    }\]
  commutes for all $d\geq0$.

  \item If $(R,\iota,\nabla_{R,d})$ satisfies the Frobenius
    congruence, then $\psi_R(x)\equiv x^{p^h} \mod \m R$.
  \end{enumerate}
\end{prop}
\begin{proof}
  The identity (1) is immediate using
  \eqref{prop:pth-power-representing-properties}.
  The square (2) is the perimeter of
    \[\xymatrix@C=10pt@R=15pt{
    {R} \ar[rr]^{\nabla_{R,d}} \ar[dr]^{\nabla_{R,d+h}} \ar[dd]_{\nabla_{R,h}}
    && {\O_d\lrtensor{s}{\O}{\iota} R} \ar[dd]^{\id\otimes\nabla_{R,h}}
    \\
    & {\O_{d+h}\lrtensor{s}{\O}{\iota} R} \ar[dr]^{\nabla_{d,h}\otimes \id}
    \ar[dd]^{\nabla_{h,d}\otimes \id}
    \\
    {\O_h\lrtensor{s}{\O}{\iota} R} \ar[dr]^{\id\otimes \nabla_{R,d}}
    \ar[dd]_{q\otimes \id}
    && {\O_d\lrtensor{s}{\O}{t} \O_h \lrtensor{s}{\O}{\iota} R}
    \ar[dd]^{\psi_{\O_d}\otimes q\otimes \id}
    \\
    & {\O_h\lrtensor{s}{\O}{t} \O_d \lrtensor{s}{\O}{\iota} R} \ar[dr]^{q\otimes
      \id\otimes \id}
    \\
    {R} \ar[rr]_{\nabla_{R,d}}
    && {\O_d\lrtensor{s}{\O}{\iota} R}
  }\]
The top and left cells commute since $(R,\iota_R,\nabla_{R,d})$ is a sheaf of
rings.  The bottom cell trivially commutes.  The right cell
commutes by \eqref{prop:pth-power-representing-properties}(3).
To prove statement (3), consider
\[\xymatrix{
    {R} \ar[r]^-{\nabla_{R,h}} \ar[d]_{\mr{proj}}
    & {\O_h\lrtensor{s}{\O}{\iota} R} \ar[d]^{\nu_h\otimes\id}
    \ar[r]^-{q\otimes \id} 
    & {\O\lrtensor{s}{\O}{\iota} R}  \ar[d]
  \\
  {R/p} \ar[r]_-{x\mapsto x^{p^h}}
  & {\O/p\,\lrtensor{s}{\O}{\iota} R} \ar[r] 
  & {\O/\m\,\lrtensor{s}{\O}{\iota} R}
  }\]
Commutativity of the left-hand square is the Frobenius congruence,
while commutativity of the right-hand square is  immediate from
\eqref{prop:nu-properties}(3).   
\end{proof}

\begin{exam}\label{exam:adams-op-on-unit}
  The Adams operation on the unit object $\ms{O}=(\O, \id, t_d)$  is
  given by $(q\otimes \id)t_d=\psi_\O$.  That is, the Adams operation
  on $\ms{O}$ coincides with the stucture map $\psi_\O$ defined in
  \eqref{prop:pth-power-representing-properties}.  In particular, it
  is an isomorphim.
\end{exam}

\begin{rem}
  We can interpret \eqref{prop:psi-properties} in the following way:
  given a 
  quasicoherent sheaf $\ms{R}=(R,\iota_R,\nabla_{R,d})$, we can define a new
  quasicoherent sheaf $\Sigma^*\ms{R}$ by basechange along the functor
  $\Sigma$.  This has the form
  \[
  \Sigma^*\ms{R} = (R, \iota_R\psi_{\O},  (\psi_{\O_d}\otimes \id)\nabla_{R,d}).
\]
The natural transformation $[p]\colon \Id\ra \Sigma$, induces a map
$\Sigma^*\ms{R}\ra \ms{R}$ of quasicoherent sheaves, which is realized
by the operator $\psi_R$.  On trivial deformations the map $[p]$ is
necessarily a height $h$ Frobenius, giving a congruence for $\psi_R$
when $\ms{R}$ satisfies the Frobenius congruence.
\end{rem}

\begin{rem}
Since $\pi_0B$ of any $B\in \CAlg_E(\Sp_{K(h)})$ is a $\T$-algebra,
and $\T$-algebras have an underlying quasicoherent sheaf satisfying
the Frobenius congruence, this
construction provides a natural operation $\psi\colon \pi_0B\ra
\pi_0B$ for any $K(h)$-local commutative $E$-algebra $B$.  This
operation can be regarded as a $p$th power \dfn{Adams operation}, and
in fact specializes to the usual one when $E$ is $p$-completed complex
$K$-theory.   
\end{rem}

\section{Frobenius lifts in quasicoherent sheaves}
\label{sec:frob-lift-qcoh}

Observe that, for a height $h$ formal group $\Gamma/\kappa$ over a
perfect field, every quasicoherent sheaf of rings $(R, \iota_R, \nabla_R)$ on
$\cDef$ naturally gives rise to an object $(R, \m_R,
\psi_R)$ of $\mc{C}_{p^h}$, using the operator $\psi_R$ defined above.
We say that   $(R,\iota_R,\nabla_R)$ is \dfn{complete} if it is so as
an object of $\mc{C}_{p^h}$, i.e., if $R\xra{\sim} \lim_k R/\m$.  
We say that it is \dfn{perfect} if it is perfect
as an object of $\mc{C}_{p^h}$, i.e., if $\psi_R$ is an isomorphism
such that $\psi_R(\m_R)=\m_R)$.

\begin{exam}
  The initial sheaf $\ms{O}=(\O,\id,t_d)$ gives  $(\O, \m,\psi_\O)$ in
  $\mc{C}_{p^h}$ which is complete and perfect.
\end{exam}

One other example will be important.
\begin{exam}
  For each $d\geq 0$, the object $(\O_d,\m_d,\psi_{\O_d})$ is a 
  complete and perfect object of $\mc{C}_{p^h}$.  
\end{exam}

\begin{thm}\label{thm:lifting-in-qcoh}
  Let $(R,\iota_R,\nabla_{R,d})$ and $(S,\iota_S,\nabla_{S,d})$
  objects of $\QCoh(\cDef)$ which satisfy the Frobenius congruence,
  and 
  suppose that $(S, \m S,
  \psi_{S})$ is complete and perfect as an
  object of $\mc{C}_{p^h}$.  
  Then for any $\O$-algebra map $\mu\colon R\ra S$, there exists a
  unique map of quasicoherent sheaves  $\wh{\mu}\colon R\ra S$ such
  that $\wh\mu\equiv 
  \mu$ modulo $\m S$.
\end{thm}
\begin{proof}
Since $(S,\m S, \psi_S)$ is complete and perfect, the pair of objects
$\iota_R\colon 
(\O,\m,\psi_{\O}) \ra (R,\m R,\psi_R)$ and 
$\iota_S\colon (\O,\m,\psi_{\O})\ra (S,\m S, 
\psi_S)$ in $\mc{C}_{p^h}$ satisfy the hypotheses of Frobenius lifting
\eqref{prop:frobenius-lifting}.  Thus there exists a unique
$\O$-algebra map $\wh\mu\colon R\ra S$ such 
that 
\begin{equation}\label{eq:psi-commute}
\vcenter{\xymatrix{
  {R} \ar[r]^{\wh\mu} \ar[d]_{\psi_R}
  & {S} \ar[d]^{\psi_{S}}
  \\
  {R} \ar[r]_{\wh\mu}
  & {S}
}}
\end{equation}
commutes, and such that $\wh\mu\equiv \mu\mod \m S$.

Next consider the 
possibly non-commutative diagram
\[\xymatrix{
  {R} \ar[r]^-{\nabla_{R,d}} \ar[d]_{\mu}
  & {\O_d \lrtensor{s}{\O}{\iota_R} R} \ar[d]^{\id\otimes \mu}
  \\
  {S} \ar[r]_-{\nabla_{S,d}} 
  & {\O_d\lrtensor{s}{\O}{\iota_S} S}
}\]
for any $d\geq 0$.
Let  $\alpha=(\id_{\O_d}\otimes \mu)\nabla_{R,d}$ and $\beta=\nabla_{S,d}\mu$ be
the two composite maps $R\ra \O_d\ltensor{s}{\O} S$ around the sides of
the square.  The proposition will follow once we can show that
$\alpha=\beta$.  We will prove this using another application of Frobenius lifting.

First note that both $\alpha$ and $\beta$ are $\O$-algebra
maps, since
\[
\alpha\iota_R = (\id_{\O_d}\otimes \mu)\nabla_{R,d}\iota_R =
(\id_{\O_d}\otimes \mu)(\id_{\O_d}\otimes\iota_{R,d})t_d =
(\id_{\O_d}\otimes \iota_S)t_d = \nabla_{S,d}\iota_S
=\nabla_{S,d}\mu\iota_R = \beta\iota_R.
\]

Next I claim that
\[
(\psi_{\O_d}\otimes \psi_{S})\alpha = \alpha\psi_R,\qquad (\psi_{\O_d}\otimes
\psi_{S})\beta=\beta\psi_R. 
\]
In the first case,
\begin{align*}
  (\psi_{\O_d}\otimes \psi_{S})\alpha
  &= (\psi_{\O_d}\otimes\psi_S)(\id_{\O_d}\otimes \mu)\nabla_{R,d}
  \\
  &= (\id_{\O_d}\otimes \mu)(\psi_{\O_d}\otimes \psi_R)\nabla_{R,d}
  \\
  &= (\id_{\O_d}\otimes \mu) \nabla_{R,d} \psi_R
    & \text{by \eqref{prop:psi-properties},}
  \\
  &= \alpha\psi_R.
\end{align*}
In the second case
\begin{align*}
  (\psi_{\O_d}\otimes \psi_{S})\beta
  &= (\psi_{\O_d}\otimes \psi_{S})\nabla_{S,d} \mu
  \\
  &=\nabla_{S,d}\psi_{S} \mu
    & \text{by \eqref{prop:psi-properties},}
  \\
  &= \nabla_{S,d}\mu\psi_R
    & \text{by \eqref{eq:psi-commute},}
  \\
  &= \beta\psi_R.
\end{align*}

Consider the ideal $\m_{d,S}= \m_d\lrtensor{s}{\O}{\iota}S \subseteq
\O_d\lrtensor{s}{\O}{\iota} S$.  
Since $s_d(\m)\subseteq \m_d$, it follows that
that
\[
(\psi_{\O_d}\otimes \psi_S)(x)\equiv x^{p^h} \mod \m_{d,S}.
\]

Consider the ideal $\m_S =
\O_d\lrtensor{s}{\O}{\iota} \m S\subseteq \O_d\lrtensor{s}{\O}{\iota}
S$.  Recall that $s\colon \O\ra \O_d$ presents $\O_d$ as a finitely
generated and free $\O$-module.  Since $S$ is complete with respect to
$\m$, it follows that $\O_d\lrtensor{s}{\O}{\iota} S$ is complete with
respect to $\m_S$.  

Since $s\colon \O\ra \O_d$ presents $\O_d$ as a finitely generated
free $\O$-module \eqref{prop:def-is-aff-gr-cat-scheme}, the quotient 
$\O_d/s(\m)$ is an artinian ring.  Thus we have that $s(\m)\subseteq
\m_d \subseteq s(\m)^N$ for 
some integer $N$, and thus $\m_S\subseteq \m_{d,S}\subseteq \m_S^N$.  
It follows that $\O_d\lrtensor{s}{\O}{\iota} S$ is also complete with
respect to $\m_{S,d}$.  

Thus, the objects $\iota_R\colon (\O,\m,\psi_O)\ra
(R,\m R,\psi_R)$ and $\nabla_{S,d}\iota_S\colon (\O,\m S,\psi_S)\ra
(\O_d\lrtensor{s}{\O}{\iota} S, \m_{d,S}, \psi_{\O_d}\otimes \psi_S)$
of $\mc{C}_{p^h}$ satisfies hypotheses of Frobenius lifting
\eqref{prop:frobenius-lifting}, which applies to show that
$\alpha=\beta$, as desired.
\end{proof}

\section{Proof of the Theorem}
\label{sec:proof-thm} 

We recall from \cite{rezk-congruence-condition} that there are 
forgetful functors
\[
\Alg_\T\ra \QCoh(\cDef)\ra \Alg_\O.
\]
These functors are monadic and comonadic
\cite{rezk-congruence-condition}*{4.22}, and so in particular preserve
small 
limits and colimits.  Furthermore, the image of
$\Alg_\T$ in $\QCoh(\cDef)$ is
contained in the full subcategory $\QCoh^{\Fr}(\cDef)$ of objects
which satisfy the Frobenius congruence.

\begin{rem}
The paper \cite{rezk-congruence-condition} actually gives fully
graded versions of these.  Here we extract only the degree 0 part of
the story.  Note that in that paper, the forgetful functor is
described by factoring through an equivalence $\Alg_{\Gamma}\approx
\QCoh(\cDef)$, where $\Alg_{\Gamma}$ describes algebras over an
explicit monad  which maps to $\T$, which we here call $\mc{Q}$.
\end{rem}

\begin{lemma}\label{lemma:t-alg-maps-tors-free}
  The map
  \[
  \Hom_{\Alg_\T}(R,S) \ra \Hom_{\QCoh(\cDef)}(R,S)
\]
induced by the forgetful functor $\Alg_\T\ra \QCoh(\cDef)$ is a
bijection whenever the underlying ring of $S$ is $p$-torsion free.  
\end{lemma}
\begin{proof}
  This is implicit (but apparently not made explicit) in
  \cite{rezk-congruence-condition}.  Both forgetful functors
  $\Alg_\T\ra \CAlg_\O$ and $\QCoh(\cDef)\ra \CAlg_\O$ are monadic, and
  the induced map $\mc{Q}\ra \T$ on monads becomes an isomorphism
  after tensoring with $\Q$.  See
  \cite{rezk-congruence-condition}*{7.4, 8.4}.  (That $\QCoh(\cDef)\ra
  \CAlg_O$ is monadic follows from
  \cite{rezk-congruence-condition}*{Thm.\ B}. 
\end{proof}

Any object $R$ of $\Alg_\T$ determines an underlying quasicoherent
sheaf $(R,\iota_R,\nabla_{R,d})$, and therefore an underlying object
$(R,\m R, \psi_R)$ in $\mc{C}_{p^h}$.

\begin{thm}\label{thm:t-alg-to-complete-and-perfect}
  Let $R$ and $S$ be objects of $\Alg_\T$, and suppose that
  \begin{enumerate}
  \item the underlying ring of $S$ is $p$-torsion free, and
  \item $(S,\m R, \psi_S)$ is complete and perfect in $\mc{C}_{p^h}$.
  \end{enumerate}
  Then the forgetful functor $\Alg_\T\ra \Alg_\O$ and the projection
  $\O\ra \kappa$ induce a bijection
  \[
  \Hom_{\Alg_\T}(R, S) \xra{\sim} \Hom_{\Alg_\O}(R, S/\m S).
  \]
\end{thm}
\begin{proof}
  Fix $S$, and let $\mc{X}\subseteq \Alg_\T$ be the collection of
  $\T$-algebras $R$ 
  such that the conclusion holds.  Since $\Alg_\T\ra \Alg_\O$
  preserves colimits, $\mc{X}$ is stable under colimits.  Every
  $\T$-algebra is a coequalizer of free $\T$-algebras, and free
  $\T$-algebras are polynomial rings
  \cite{strickland-morava-e-theory-of-symmetric}*{\S5}.  Thus, we can reduce
  to 
  the case that the underlying $\O$-algebra of $R$ is polynomial. 

  In this case, every $\O$-algebra map $\ol\mu\colon R\ra
  S/\m S$ lifts to an $\O$-algebra map $\mu\colon R\ra S$.  By
  \eqref{thm:lifting-in-qcoh}, there exists a unique morphism
  $\wh{\mu}\colon R\ra S$ in 
  $\QCoh(\cDef)$ such that $\wh{\mu}\equiv \mu\mod \m S$.  By
  \eqref{lemma:t-alg-maps-tors-free}
  this $\wh{\mu}$ is the unique morphism in $\Alg_\T$ with this
  property.  
\end{proof}

\begin{proof}[Proof of \eqref{thm:cofreeness-of-O}]
  Immediate from \eqref{thm:t-alg-to-complete-and-perfect}, as the
  $\T$-algebra $\pi_0E = \O$ is torsion free, complete and perfect
  \eqref{exam:adams-op-on-unit}. 

\end{proof}

Let $E\colon \Perf_\kappa\ra \CAlg(\Sp_{K(h)})$ denote the functorial
extension of Morava $E$-theory described in
\cite{burklund-schlank-yuan-chromatic-nullstellensatz}*{\S2}.  We write
$\O(A)=\pi_0E(A)$.

\begin{proof}[Proof of \eqref{thm:cofreeness-of-OA}]
  It suffices to show that $\O(A)$ is $p$-torsion free,
  complete, and perfect.   We have
  \cite{burklund-schlank-yuan-chromatic-nullstellensatz}*{Thm.\ 2.38}
  \[
  \pi_0E(A) \approx \Witt_pA \powser{u_1,\dots, u_{h-1}}, \qquad \m
  \pi_0E(A) = (p,u_1,\dots,u_{h-1}),
\]
which shows that $\O(A)$ is $p$-torsion free and complete with
respect to $\m$, and that $\O(A)/\m \O(A)\approx A$.
This also provides an isomorphism of rings
\[
\bigl[ \Witt_p(A)\otimes_{\Witt_p(\kappa)} \O\bigr]^{\eev}_p \xra{\sim} \O(A).
\]
It remains to show that the Adams operation $\psi_{\O(A)}$ (induced by
the natural 
$\T$-algebra structure on $\pi_0E(A)$) is an isomorphism which
preserves $\m \O(A)$.  This is immediate from the following claim.

\textit{Claim.}  $\psi_{\O(A)}$ is the $p$-completion of
$\wt\phi^h\otimes \psi_\O$, where $\wt\phi\colon \Witt_p(A)\ra
\Witt_p(A)$ is the lift of $p$th power Frobenius.

\textit{Proof of Claim.}
By naturality, we know that the action of $\psi_{O(A)}$ is compatible
with that of $\psi_\O$ under the $\O$-algebra structure map $\O\ra
\O(A)$.  Thus, it suffices to compute the composite of $\Witt_pA\ra
\O(A)\xra{\psi_{\O(A)}} \O(A)$.  By
  \cite{burklund-schlank-yuan-chromatic-nullstellensatz}*{Lemma
    3.27}, there is an equivalence
  \[
  \Witt_\Sphere(A) \otimes_{\Witt_\Sphere(\kappa)} E \xra{\approx} E
  \]
in $\CAlg(\Sp_{K(h)})$, where $\Witt_\Sphere(A)$ denotes the
\emph{spherical Witt vectors} \cite{lurie-ell2}
\cite{burklund-schlank-yuan-chromatic-nullstellensatz}*{\S2.1}.
The spherical Witt vectors are an object of $\CAlg(\Sp_p)_{\geq0}$,
whose underlying spectrum is a Moore spectrum of the abelian group
$\Witt_pA$, so $\pi_0\Witt_\Sphere(A)=\Witt_p(A)$.

We have a commutative diagram 
\[\xymatrix{
  {\pi_0 \Witt_\Sphere A} \ar[r]^-{P^{p^h}} \ar[d]
  & {\pi_0\ul\Map(\Sigma^\infty_+ B\Sigma_{p^h}, \Witt_\Sphere A)} \ar[d]
  \\
  {\pi_0 E(A)}  \ar[r]_-{P^{p^h}}
  & {\pi_0\ul\Map(\Sigma^\infty_+, B\Sigma_{p^h}, E(A))} \ar[r]
  & {\O_h\otimes_\O \O(A)} \ar[r]_-{\nu_h\otimes\id}
  & {\O(A)} 
}\]
where $P^m\colon \pi_0 R \ra \pi_0\ul\Map(\Sigma^\infty_+ B\Sigma_m,
R)$ denotes the $m$th power operation of an $R\in \CAlg(\Sp)$, and the
other maps are ring homomorphisms.  
The Adams operation $\psi_{\O(A)}$ is the composite of  the bottom
row. 
Since the spherical Witt vectors
are connective, for any $m$ we have
$\pi_0\ul\Map(\Sigma^\infty_+B\Sigma_m, \Witt_\Sphere A) =
\pi_0\Witt_\Sphere A$ and  $P^m(c)=c^m$.
\end{proof}


\end{document}